\documentclass{article}
\usepackage[utf8]{inputenc}
\usepackage{pifont}
\usepackage{amssymb, amsmath, amsthm, amsfonts,enumerate}
\usepackage{amsmath,setspace,scalefnt}
\usepackage[usenames,dvipsnames,svgnames,table]{xcolor}
\usepackage{graphicx,tikz,caption,subcaption}
\usetikzlibrary{patterns}
\usetikzlibrary{shapes}
\usepackage{fullpage}
\usepackage{hyperref}
\usepackage{enumitem,verbatim}%,itemize
\usetikzlibrary{decorations.pathmorphing}
\usepackage{subcaption}
\usepackage{pgfplots}
\usepackage{mathrsfs}
\usetikzlibrary{arrows}

\tikzset{snake it/.style={decorate, decoration=snake}}
\usetikzlibrary{calc}

\definecolor{gray}{rgb}{0.25, 0.25, 0.25}

\usetikzlibrary{decorations.markings}%hongbao

\usetikzlibrary{matrix,arrows,decorations.pathmorphing}
\usetikzlibrary{positioning}
\usetikzlibrary{graphs} % For \graph
\usetikzlibrary{graphs.standard}

\newtheorem{theorem}{Theorem}[section]
\newtheorem{corollary}[theorem]{Corollary}
\newtheorem{lemma}[theorem]{Lemma}

\newtheorem{claim}{Claim}

\newtheorem{conjecture}[theorem]{Conjecture}

\theoremstyle{definition}

\newtheorem{definition}{Definition}

\title{A short note on spanning even trees}

\author{Jiangdong Ai\thanks{Corresponding author.  School of Mathematical Sciences and LPMC, Nankai University, Tianjin 300071, P.R.
China. Email: jd@nankai.edu.cn. },~ Zhipeng Gao\thanks{School of Mathematics and Statistics, Xidian University, Xi'an 710126, P.R.
China. Email: gaozhipeng@xidian.edu.cn.}, Xiangzhou Liu\thanks{Department of Mathematics, Tiangong University, Tianjin 300071, P.R.
China. Email: i19991210@163.com.}, Jun Yue\thanks{Department of Mathematics, Tiangong University, Tianjin 300071, P.R.
China. Emailyuejun06@126.com}}
\begin{document}

\maketitle
\begin{abstract}
We call a tree $T$ is \emph{even} if every pair of its leaves is joined by a path of even length. Jackson and Yoshimoto~[J. Graph Theory, 2024] conjectured that every $r$-regular nonbipartite connected graph $G$ has a spanning even tree. They verified this conjecture for the case when $G$ has a $2$-factor. In this paper, we prove that the conjecture holds when $r$ is odd, thereby resolving the only remaining unsolved case for this conjecture. 
\end{abstract}

% \smallskip
% \noindent{\bf Keywords:}
% Multi-partite tournament; Strong partition; Control\\

\section{Introduction}
We refer to \cite{bondy2008graph} for terminology and notation not introduced here. All graphs considered throughout this paper contain no loops. A tree is called \emph{even} if all pairs of vertices of degree one are joined by a path of even length. For a graph $G$, we call a subgraph $H$, a \emph{spanning} subgraph, of $G$ if $V(H)=V(G)$. In particular,  we say $H$ is a $k$-factor of $G$ if it is a $k$-regular spanning subgraph of $G$. 

Motivated by Saito's question regarding the existence of spanning even trees in regular connected graphs, Jackson and Yoshimoto proposed the following conjecture:
\begin{conjecture}\cite{jacksonspanning}\label{JYC}
    If $G$ is an $r$-regular nonbipartite connected graph, then $G$ contains a spanning even tree.
\end{conjecture}
This conjecture was partially confirmed ([Theorem 1] in \cite{jacksonspanning}) under the condition that $G$ possesses a 2-factor. So sufficient conditions for an regular graph to have a 2-factor will be helpful. Especially, Petersen \cite{petersen1891theorie} showed that every $r$-regular graph where $r$ is even admits a 2-factor, narrowing our focus to verifying the conjecture for odd integer $r$.
\begin{theorem}\cite{jacksonspanning}
    Every regular nonbipartite connected graph which has a 2-factor has a spanning even tree.
\end{theorem}
In fact, following a conclusion established by Hanson et al. \cite{hanson1998interval}, the authors \cite{jacksonspanning} posed that the only remaining unsolved case is when $r$ is odd and $G$ has at least $r$ bridges.  

Following the same idea, we hope to find a 2-factor in an $r$-regular graph where $r$ is odd. However, this does not always come true since we can easily obtain a counter-example as depicted in Figure \ref{pp-1}. It's reasonable to assume that there exists a spanning subgraph similar to `2-factor' in a regular graph, that is $\{1,2\}$-factor, which is defined in Definition \ref{def}. There is a huge literature investigating graph factors and factorization, see more information on \cite{plummer2007graph}.

In this short note, we address the final unsolved case for Conjecture~\ref{JYC}.

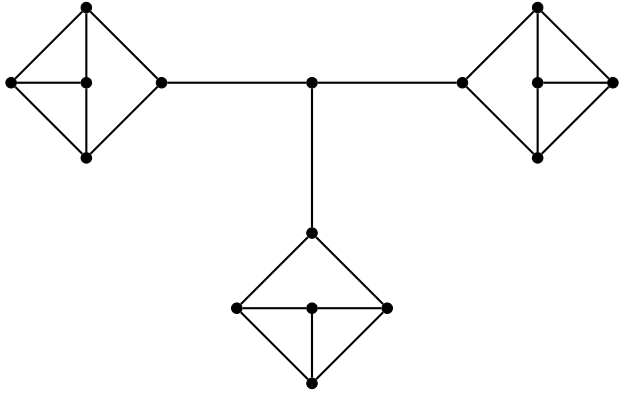
\begin{figure}[h]
		\centering
		\tikzset{  
			mynode/.style={  
				draw,
				circle,
				inner sep=1.4pt,
				fill
			}  
		}  
		\begin{tikzpicture}		
    \node [mynode](a) at (0,0){};
\node [mynode](b1) at (2,0){} ;
\node [mynode](b2) at (3,1){} ;
\node [mynode] (b3) at (4,0){};
\node [mynode] (b4) at (3,-1){};
\node [mynode] (b5) at (3,0) {};
\draw[thick] (a)--(b1)--(b2)--(b3)--(b4)--(b1);
\draw [thick] (b3)--(b5)--(b4);
\draw[thick] (b5)--(b2);
\node [mynode] (c1) at (0,-2){};
\node [mynode] (c2) at (-1,-3){};
\node [mynode] (c3) at (1,-3){};
\node [mynode] (c5) at (0,-4){};
\node [mynode] (c4) at (0,-3){};
\draw[thick] (a)--(c1)--(c2)--(c5)--(c3)--(c1);
 \draw[thick] (c4)--(c5);
 \draw[thick] (c2)--(c4)--(c3);
\node [mynode] (d1) at (-3,0){} ;
\node [mynode] (d2) at (-2,0) {};
\node [mynode] (d3) at (-3,1){} ;
\node [mynode] (d4) at (-4,0){} ;
\node [mynode] (d5) at (-3,-1){};
\draw[thick] (a)--(d2)--(d3)--(d4)--(d5)--(d2);
 \draw[thick] (d1)--(d4);
 \draw[thick] (d3)--(d1)--(d5);
				% \draw[thick] (a)node[above]{$x_i$}--(b)node[below]{$y_i$}--(c)node[above]{$z_i$};
				% \node[mynode](u1) at (-1.25,-1.5){};
				% %	\node[mynode](v1) at (1.25,-1.5){};
				%\node[mynode](w2) at (-.75,-1.5){};

		\end{tikzpicture}
		\caption{A 3-regular graph without 2-factor.}
		\label{pp-1}
	\end{figure}

\noindent {\bf Additional Definitions and Useful Tools.}
\begin{definition}
A spanning subgraph $F$ of $G$ is called an $\{a,b\}$-factor of $G$ if $a\leq d_{F}(x) \leq b$ for all $x\in V(G)$, and we usually call a $\{k,k\}$-factor a $k$-factor. A $\{k-1,k\}$-factor is called a good $\{k-1,k\}$-factor if each of whose components is regular.
\end{definition}

\begin{definition}\label{def}
An even tree $T$ in a graph $G$ is said to be good if $G\setminus V(T)$ has a good $\{1,2\}$-factor.
\end{definition}

\begin{theorem}\label{1,2factor}\cite{kano1986factors}
Let $r$ and $k$ be positive integers. If $k\leq 2(2r+1)/3$, then every $(2r+1)$-regular graph has a good $\{k-1,k\}$-factor.
\end{theorem}

\begin{lemma}\label{J.Petersen}\cite{petersen1891theorie}
    Every $2k$-regular graph contains a $2$-factor.
\end{lemma}

% \begin{lemma}\label{[1,2]factor}
%  Let $G$ be an $r$-regular graph where $r$ is odd and $r>1$, then $G$ has a good $[1,2]$-factor.   
% \end{lemma}

% \begin{proof}
% Since $r$ is odd and $r>1$, by theorem\ref{1,2factor} we have that $G$ has a good $[1,2]$-factor.  
% \end{proof}

\begin{corollary}\label{coro1}
    Let $G$ be an $r$-regular connected graph where $r$ is odd, then it has a good $\{1,2\}$-factor $F$ with at least one $2$-regular component. 
\end{corollary}
\begin{proof}
     Let $k=2$, then there is a good $\{1,2\}$-factor $F$ by Theorem \ref{1,2factor}. If $F$ is not a perfect matching, the result holds clearly.
     If $F$ is a perfect matching, we can derive a spanning subgraph $G^{'}$ of $G$ by deleting $E(F)$. Since $G^{'}$ is $(r-1)$-regular, Lemma \ref{J.Petersen} ensures the existence of a 2-factor of $G^{'}$, which is also a 2-factor of $G$.

\end{proof}

In~\cite{jacksonspanning}, there is a simple but valuable lemma.
\begin{lemma}\label{fk}\cite{jacksonspanning}
Let $G$ be an $r$-regular nonbipartite connected graph and $Y\subset V(G)$ with $E(Y)=\emptyset$. Then for any $X\subset V(G)\setminus Y$ with $|X|\leq |Y|, E(Y,V(G)\setminus (X\cup Y))\neq \emptyset$. 
\end{lemma}

\section{Main Results}
We want to point out that the main proof technique employed in this section is significantly similar to that of Theorem $4$ in~\cite{jacksonspanning} of Jackson and Yoshimoto. Compared with their proofs, we work on the $\{1,2\}$-factor rather than the $2$-factor. This adjustment necessitates an additional discussion regarding the component that consists of a single edge. In the following proof, denote by $[a,b]$ the set of $\{a,a+1,a+2,\ldots,b-1,b\}$. In particular, let $[b]=[1,b]$.

\begin{theorem}\label{good}
 If $G$ is an $r$-regular nonbipartite connected graph where $r$ is odd, then $G$ has a good even tree.

\end{theorem} 

\begin{proof}
    By Corollary \ref{coro1} there is a good $\{1,2\}$-factor $F=F_{1}\cup F_{2}\cup\cdots \cup F_{k}$ of $G$ such that $F_i$ is a cycle for some $i$. Without loss of generality, we assume that $F_{1}$ is a cycle. For any vertex $v$ of $F_i$ ($i\in [k]$), 
    we define two functions as follows: $$e_{1}(v,F_{i}) = 
\begin{cases} 
uv,u\in N_{F_{i}}(v) & \text{if $F_{i}$ is a cycle,} \\
\emptyset & \text{if $F_{i}$ is an edge.} 
\end{cases}$$
$$
e_{2}(v,F_{i}) = 
\begin{cases} 
u_{1}v\cup u_{2}v,u_{1},u_{2}\in N_{F_{i}}(v) & \text{if $F_{i}$ is a cycle,} \\
E(F_{i}) & \text{if $F_{i}$ is an edge.} 
\end{cases}
$$ 
    
    If $F_1$ is an odd cycle, then $T=F_1\setminus e_{1}(x,F_1)$ is a good even tree for any vertex $x$ of $F_1$. Thus, we may assume that there is no odd cycle in $F$. Let $(X_{1},Y_{1})$ be a bipartition of $F_1$. Since $F_1$ is an even cycle, we have $|X_{1}|=|Y_{1}|$. If there is a chord $uv\in E(G)$ such that $u,v\in Y_{1}$, then $T=(F_{1}\setminus e_{2}(v,F_{1}))\cup \{uv\}$ is a good even tree. Therefore, we may assume that $E(Y_{1})=\emptyset$. Since $|X_{1}|=|Y_{1}|$, by Lemma \ref{fk} we have that $E(Y_{1},V(G)\setminus V(F_{1}))\ne \emptyset$.

   We say that a bipartition $(X,Y)$ of $F$ is consistent if $X_{1}\subseteq X$ and $Y_{1}\subseteq Y$. For any consistent bipartition of $F$, let $X_{i}=X\cap V(F_{i})$ and $Y_{i}=Y\cap V(F_{i})$ for all $i\in [2,k]$. Then $(X_{i},Y_{i})$ is a bipartition of $F_{i}$ for all $i\in [k]$. A sequence $F_{i_{1}},F_{i_{2}},\cdots ,F_{i_{t}}$ in $F$ is a good sequence with respect to the bipartition $(X,Y)$ if $E(Y_{i_{h}},X_{i_{h+1}})\ne \emptyset$ for all $1\leq h <t$. We say that $F_{i}$ is good with respect to the bipartition $(X,Y)$ if there is a good sequence from $F_1$ to $F_{i}$.

   Now we choose a consistent bipartition $(X,Y)$ of $F$ such that the number of $F_{i}$ which is good with respect to $(X,Y)$ is as large as possible. %Let $c$ be the set of all factors in $F^{even}$ which are good with respect to $(X,Y)$. 
   By symmetry, we may assume $F^{good}=\{F_{1},F_{2},\cdots ,F_{t}\}$. \\

   {\bf Case 1. $E(\cup ^{t}_{i=1} Y_{i})=\emptyset$.}
\\

   Since $\sum_{i=1}^{t}|X_{i}|=\sum_{i=1}^{t}|Y_{i}|$, Lemma \ref{fk} implies that $E(\cup ^{t}_{i=1} Y_{i},V(G)\setminus\cup_{i=1}^{t}F_{i})\ne \emptyset$. Then there is a vertex $v\in F_{j}$ such that $E(\cup ^{t}_{i=1} Y_{i},v) \ne \emptyset$, where $j>t$. Choose a bipartition $(X_{j}^{'},Y_{j}^{'})$ of $F_{j}$ such that $E(\cup ^{t}_{i=1} Y_{i},X_{j}^{'}) \ne \emptyset$ (By choosing $X_{j}^{'}$ such that $v\in X_{j}^{'}$). Then $F_{j}$ becomes to be good with respect to the bipartition $X^{'}=(X\setminus X_{j})\cup X_{j}^{'}$ and $Y^{'}=(Y\setminus Y_{j})\cup Y_{j}^{'}$ of $F$ and all elements in $F^{good}\cup \{F_{j}\}$ remain good with respect to $(X^{'},Y^{'})$. This contradicts the choice of $(X,Y)$.\\

    {\bf Case 2. $E(\cup ^{t}_{i=1} Y_{i})\neq \emptyset$.}
\\

    We can choose an edge $y_{p}y_{q}\in E(Y_{p},Y_{q})$ for some $1 \leq p \leq q\leq k$. Let $$F_{i_{1}},F_{i_{2}},\cdots ,F_{i_{r}}  \ \ \text{and} \ \ 
 F_{j_{1}},F_{j_{2}},\cdots ,F_{j_{s}}$$ be good sequences with respect to $(X,Y)$ with $i_{1}=1=j_{1}$, $i_{r}=p$ and $j_{s}=q$.  Choose an edge $y_{i_{h}}x_{i_{h+1}}\in E(Y_{i_{h}},X_{i_{h+1}})$ for all $1\leq h<r$ and an edge $y_{j_{h}}x_{j_{h+1}}\in E(Y_{j_{h}},X_{j_{h+1}})$ for all $1\leq h<s$. Consider the following two subcases.\\   

{\bf Subcase 2.1. $F_{q}\in \{F_{i_{1}},F_{i_{2}},\cdots ,F_{i_{r}}\}$.}
\\

We have $F_{q}=F_{i_{a}}$ for some $1\leq a\leq r$. Then $$T=\left(\bigcup_{1\leq h\leq {r-1}}(y_{i_{h}}x_{i_{h+1}}\cup (F_{i_{h}}\setminus e_{1}(y_{i_{h}},F_{i_{h}})))\right)\cup \{y_{p}y_{q}\} \cup (F_{p}\setminus e_{2}(y_{p},F_{p}))$$ is a good even tree.\\

{\bf Subcase 2.2. $F_{q}\notin \{F_{i_{1}},F_{i_{2}},\cdots ,F_{i_{r}}\}$.} 
\\

Choose an integer $j_{b}$ such that $F_{j_{b}}\in 
 \{F_{i_{1}},F_{i_{2}},\cdots ,F_{i_{r}}\}$ and $b$ is as large as possible. We have $1\leq b<s$ since $j_{1}=i_{1}$ and $F_{j_{s}}=F_{q}\notin \{F_{i_{1}},F_{i_{2}},\cdots ,F_{i_{r}}\}$. Then $F_{j_{b}}=F_{i_{a}}$ for some $1\leq a\leq r$ and $\{F_{i_{1}},F_{i_{2}},\cdots ,F_{i_{r}}\} \cap \{F_{j_{b+1}},F_{j_{b+2}},\cdots ,F_{j_{s}}\}= \emptyset$. Then, 

 \begin{align*}
    T= &\left(\bigcup_{1\leq h\leq {r-1}}(y_{i_{h}}x_{i_{h+1}}\cup (F_{i_{h}}\setminus e_{1}(y_{i_{h}},F_{i_{h}})))\right)\cup \{y_{p}y_{q}\} \cup (F_{p}\setminus e_{2}(y_{p},F_{p}))\cup \\ &\left(\bigcup_{b+1 \leq h\leq {s-1}}(y_{j_{h}}x_{j_{h+1}}\cup (F_{j_{h}}\setminus e_{1}(y_{j_{h}},F_{j_{h}})))\right) \cup \{y_{j_{b}}x_{j_{b+1}}\} \cup (F_{q}\setminus e_{1}(y_{q},F_{q}))
\end{align*}
is a good even tree.
\end{proof}

\begin{theorem}\label{spanning}
  If $G$ is an $r$-regular nonbipartite connected graph where $r$ is odd, then $G$ has a spanning even tree.  
\end{theorem}

\begin{proof}
Let $T$ be a good even tree such that $|V(T)|$ is as large as possible, the existence of $T$ is guaranteed by Theorem~\ref{good}. Suppose, for the sake of contradiction, that $V(T)<V(G)$. Let $(X_{0},Y_{0})$ be the  bipartition of $T$ where $X_{0}$ contains all leaves of $T$. Let  $F=F_{1}\cup F_{2}\cup\cdots \cup F_{m}$ be a good $\{1,2\}$-factor of $G\setminus V(T)$.

\begin{claim}\label{X_{0}}
    $E(X_{0},F)=\emptyset$ and $E(Y_{0},F)\neq \emptyset$.
\end{claim}
\begin{proof}
    Suppose, for the sake of contradiction, that there is an edge $xu\in E(X_{0},F_{i})$ for some $1\leq i \leq m$. If $F_{i}$ is an even cycle, then $T^{'}=T\cup \{xu\}\cup (F_{i}\setminus e_{1}(u,F_{i}))$ is a good even tree, a contradiction to the maximality of $T$. If $F_{i}=uv_{1}v_{2} \dots v_{c}u$ is an odd cycle, then $T^{'}=T\cup \{xu\}\cup (F_{i}\setminus \{v_{1}v_{2}\})$ is a good even tree, a contradiction to the maximality of $T$. If $F_{i}$ is an edge then $T^{'}=T\cup \{xu\}\cup F_{i}$ is a good even tree, a contradiction to the maximality of $T$.

    Hence $E(X_{0},F)=\emptyset$. The assertion that $E(Y_{0},F)\neq \emptyset$ now follows from the assumptions that $G$ is connected and $V(T)<V(G)$.    
\end{proof}

Thus, we suppose $yu\in E(Y_{0},F_{i})$ for some $i\in [m]$. If $F_{i}$ is an odd cycle, then $T^{'}=T\cup \{yu\}\cup (F_{i}\setminus e_{1}(u,F_{i}))$ is a good even tree, a contradiction to the maximality of $T$. Therefore, $F_{i}$ is an even cycle or an edge. 
Without loss of generality, let $F^{even}=\{F_{1},F_{2},\cdots ,F_{k}\}$ be the set of all even cycles and edges in $F$. Clearly, $F^{even}\neq \emptyset$. Let $F^{'}=\cup^{k}_{j=1}F_{j}$. Since $F^{even}$ is the set of all even cycles and edges in $F$, $F^{'}$ is a balanced bipartite subgraph of $G$. Let $(X^{'},Y^{'})$ be a bipartition of $F^{'}$, let $X_{i}=X^{'}\cap V(F_{i})$ and $Y_{i}=Y^{'}\cap V(F_{i})$ for all $i\in [k]$. Then $(X_{i},Y_{i})$ is a bipartition of $F_{i}$ for all $i\in [k]$. A sequence $F_{i_{1}},F_{i_{2}},\cdots ,F_{i_{t}}$ in $F^{even}$ is an additive sequence with respect to the bipartition $(X^{'},Y^{'})$ if $E(Y_{i_{h}},X_{i_{h+1}})\ne \emptyset$ for all $1\leq h<t$ and $E(Y_{0},X_{i_{1}})\ne \emptyset$. We say that $F_{j}\in F^{even}$ is additive with respect to the bipartition $(X^{'},Y^{'})$ if there exists an additive sequence ending at it. 

   Now we choose a bipartition $(X^{'},Y^{'})$ of $F^{'}$ such that the number of elements in $F^{even}$ which are additive with respect to $(X^{'},Y^{'})$ is as large as possible. Let $F^{add}$ be the set of elements in $F^{even}$ which are additive with respect to $(X^{'},Y^{'})$. Then by Claim \ref{X_{0}} and the definition of $F^{even}$, we obtain that $F^{add}\neq \emptyset$.
   By symmetry, we may assume $F^{add}=\{F_{1},F_{2},\cdots ,F_{t}\}$. \\

   {\bf Case 1. $E(\cup ^{t}_{i=1} Y_{i})=\emptyset$.}
\\

Lemma \ref{fk} implies that $E(\cup ^{t}_{i=1} Y_{i},V(G)\setminus V(F^{add}))\neq \emptyset$. Notice that $$V(G)\setminus V(F^{add})=X_{0}\cup Y_{0} \cup \bigcup_{j=t+1}^{m} V(F_{j})   $$ and by Claim \ref{X_{0}}, $E(X_{0},F)=\emptyset$.
If there exists an edge $y_{p}y_{0}^{'}\in E(Y_{p},Y_{0})$ for some $1\leq p\leq t$. Let  $F_{i_{1}},F_{i_{2}},\cdots ,F_{i_{s}}$ be an additive sequence with respect to $(X^{'},Y^{'})$ with $i_{s}=p$. Choose an edge $y_{i_{h}}x_{i_{h+1}}\in E(Y_{i_{h}},X_{i_{h+1}})$ for all $1\leq h < s$ and $y_{0}x_{i_{1}}\in E(Y_{0},X_{i_{1}})$. Then $$T^{'}=T\cup \{y_{0}x_{i_{1}}\} \cup \left(\bigcup^{s-1}_{h=1}(y_{i_{h}}x_{i_{h+1}}\cup (F_{i_{h}}\setminus e_{1}(y_{i_{h}},F_{i_{h}})))\right)\cup \{y_{p}y_{0}^{'}\} \cup (F_{p}\setminus e_{2}(y_{p},F_{p}))$$ is a good even tree in $G$. This contradicts the maximality of $T$.

Hence there is an edge $y_{p}v_{j}\in E(Y_{p},F_{j})$ for some $1 \leq p \leq t$ and some $t<j\leq m$. If $F_{j}$ is an even cycle or an edge, then $F_{j}$ is additive with respect to the bipartition $X^{''}=(X^{'}\setminus X_{j})\cup Y_{j}$ and $Y^{''}=(Y^{'}\setminus Y_{j})\cup X_{j}$ and all elements in $F^{add}\cup \{F_{j}\}$ remain additive with respect to $(X^{''},Y^{''})$. This contradicts the choice of $(X^{'},Y^{'})$. Thus, $F_{j}$ must be an odd cycle. Then $$T^{'}=T\cup \{y_{0}x_{i_{1}}\}\cup \left(\bigcup^{s-1}_{h=1}(y_{i_{h}}x_{i_{h+1}}\cup (F_{i_{h}}\setminus e_{1}(y_{i_{h}},F_{i_{h}})))\right) \cup \{y_{p}v_{j}\} \cup (F_{j}\setminus e_{1}(v_{j},F_{j}))\cup (F_{p}\setminus e_{1}(y_{p},F_{p}))$$ is a good even tree in $G$. This contradicts the maximality of $T$.\\

   {\bf Case 2. $E(\cup ^{t}_{i=1} Y_{i})\neq \emptyset$.}
\\

    We can choose an edge $y_{p}y_{q}^{'}\in E(Y_{p},Y_{q})$ for some $1 \leq p \leq q\leq t$. Let $$F_{i_{1}},F_{i_{2}},\cdots ,F_{i_{r}}  \ \ \text{and} \ \ 
 F_{j_{1}},F_{j_{2}},\cdots ,F_{j_{s}}$$ be additive sequences with respect to $(X^{'},Y^{'})$ with $i_{r}=p$ and $j_{s}=q$.  Choose edges $y_{0}x_{i_{1}}\in E(Y_{0},X_{i_{1}})$, $y_{0}^{'}x_{j_{1}}^{'}\in E(Y_{0},X_{j_{1}})$, $y_{i_{h}}x_{i_{h+1}}\in E(Y_{i_{h}},X_{i_{h+1}})$ for all $1\leq h<r$ and $y_{j_{h}}^{'}x_{j_{h+1}}^{'}\in E(Y_{j_{h}},X_{j_{h+1}})$ for all $1\leq h<s$. Consider the following two subcases.\\   

{\bf Subcase 2.1. $F_{q}\in \{F_{i_{1}},F_{i_{2}},\cdots ,F_{i_{r}}\}$.}
\\

    We have $F_{q}=F_{i_{a}}$ for some $1\leq a\leq r$. Then $$T^{'}=T\cup \left(\bigcup_{1\leq h\leq {r-1}}(y_{i_{h}}x_{i_{h+1}}\cup (F_{i_{h}}\setminus e_{1}(y_{i_{h}},F_{i_{h}})))\right)\cup \{y_{p}y_{q}^{'}\} \cup (F_{p}\setminus e_{2}(y_{p},F_{p}))\cup \{y_{0}x_{i_{1}}\}$$ is a good even tree. This contradicts the maximality of $T$.\\

{\bf Subcase 2.2. $F_{q}\notin \{F_{i_{1}},F_{i_{2}},\cdots ,F_{i_{r}}\}$.} 
\\

If $\{F_{i_{1}},F_{i_{2}},\cdots ,F_{i_{r}}\} \cap \{F_{j_{1}},F_{j_{2}},\cdots ,F_{j_{s}}\} \neq \emptyset$, then we choose an integer $j_{b}$ such that $F_{j_{b}}\in \{F_{i_{1}},F_{i_{2}},\cdots ,F_{i_{r}}\}$ and $b$ is as large as possible. We have $1\leq b<s$ since $F_{j_{s}}=F_{q}\notin \{F_{i_{1}},F_{i_{2}},\cdots ,F_{i_{r}}\}$. Then $F_{j_{b}}=F_{i_{a}}$ for some $1\leq a\leq r$ and $\{F_{i_{1}},F_{i_{2}},\cdots ,F_{i_{r}}\} \cap \{F_{j_{b+1}},F_{j_{b+2}},\cdots ,F_{j_{s}}\}= \emptyset$. 

Let 
\begin{align*}
    T^{'}= &T\cup \{y_{0}x_{i_{1}}\}\cup \left(\bigcup_{1\leq h\leq {r-1}}(y_{i_{h}}x_{i_{h+1}}\cup (F_{i_{h}}\setminus e_{1}(y_{i_{h}},F_{i_{h}})))\right)\cup \{y_{p}y_{q}^{'}\} \cup (F_{p}\setminus e_{2}(y_{p},F_{p}))\cup \\ &\left(\bigcup_{b+1 \leq h\leq {s-1}}(y_{j_{h}}^{'}x_{j_{h+1}}^{'}\cup (F_{j_{h}}\setminus e_{1}(y_{j_{h}},F_{j_{h}})))\right) \cup \{y_{j_{b}}^{'}x_{j_{b+1}}^{'}\} \cup (F_{q}\setminus e_{1}(y_{q}^{'},F_{q}))
\end{align*}
Then $T^{'}$ is a good even tree. This contradicts the maximality of $T$.

Therefore, $\{F_{i_{1}},F_{i_{2}},\cdots ,F_{i_{r}}\} \cap \{F_{j_{1}},F_{j_{2}},\cdots ,F_{j_{s}}\}= \emptyset$, then 

 \begin{align*}
    T^{'}= &T\cup \{y_{0}x_{i_{1}}\}\cup \left(\bigcup_{1\leq h\leq {r-1}}(y_{i_{h}}x_{i_{h+1}}\cup (F_{i_{h}}\setminus e_{1}(y_{i_{h}},F_{i_{h}})))\right)\cup \{y_{p}y_{q}^{'}\} \cup (F_{p}\setminus e_{2}(y_{p},F_{p}))\cup \\ &\left(\bigcup_{1\leq h\leq {s-1}}(y_{j_{h}}^{'}x_{j_{h+1}}^{'}\cup (F_{j_{h}}\setminus e_{1}(y_{j_{h}},F_{j_{h}})))\right) \cup \{y_{0}^{'}x_{j_{1}}^{'}\}  \cup (F_{q}\setminus e_{1}(y_{q}^{'},F_{q}))
\end{align*}
is a good even tree. This contradicts the maximality of $T$ and
completes the proof of Theorem \ref{spanning}.

\end{proof}

% \begin{thebibliography}{99}	
	
% 		\bibitem {B-G} J. Bang-Jensen and G. Gutin, Digraphs: Theory, Algorithms and Applications, Springer-Verlag, London, 2nd edition (2009).

% 	\end{thebibliography}

\bibliographystyle{plain}
\bibliography{refs} 
 
\end{document}